\documentclass[12pt, reqno]{amsart}
\usepackage{amssymb}
\input xypic
\xyoption{all}

\newcommand{\calf}{{\mathcal F}}

\newcommand{\calo}{{\mathcal O}}

\newcommand{\calu}{{\mathcal U}}

\newcommand{\cali}{{\mathcal J}}
\newcommand{\calw}{{\mathcal W}}
\newcommand{\calv}{{\mathcal V}}

\newcommand{\jono}{{\rm Jon}_o(\PP^n)}

\newcommand{\bir}{{\rm Bir}}
\newcommand{\birplane}{{\rm Bir}(\PP^2)}
\newcommand{\birtwo}{{\rm Bir}(\PP^2)}
\newcommand{\birn}{{\rm Bir}(\PP^n)}

\newcommand{\pgl}{{\rm PGL}}

\newcommand{\x}{{\bf x}}
\newcommand{\Deg}{{\rm Deg}}
\newcommand{\staro}{{\rm St}_o(\PP^n)}
\renewcommand{\star}{{\rm St}}

\newcommand{\PP}{{\mathbb P}}
\renewcommand{\P}{{\mathbb P}}

\newcommand{\A}{{\mathbb A}}

\newcommand{\C}{{\mathbb C}}

\newcommand{\N}{{\mathbb N}}

\newcommand{\Z}{{\mathbb Z}}

\newcommand{\tor}{\dashrightarrow }

\begin{document}

\title[ ]{Some remarks about the Zariski topology of the Cremona group}

\newtheorem{thm}{Theorem}
\newtheorem{pro}[thm]{Proposition}
\newtheorem{cor}[thm]{Corollary}
\newtheorem{rems}[thm]{Remarks}
\newtheorem{lem}[thm]{Lemma}
\newtheorem{defi}[thm]{Definition}
\newtheorem{con}{Conjecture}
\theoremstyle{remark}
\newtheorem{exa}[thm]{Example}
\newtheorem{exas}[thm]{Examples}

\newtheorem{rem}[thm]{Remark}

\begin{abstract}  For an algebraic variety $X$ we study the behavior of algebraic morphisms from an algebraic variety to the  
 group $\bir(X)$ of birational maps of $X$ and obtain, as application, some insight about the relationship between the so-called Zariski topology of $\bir(X)$ and the  algebraic structure of this group, in the case where $X$ is
rational.  
\end{abstract}

% \begin{center}
\author{Ivan Pan and Alvaro  Rittatore} 
\thanks{Both authors are partially supported by the ANII, MathAmSud and  CSIC-Udelar (Uruguay).}
% \end{center}
\address{Ivan Pan, Centro de Matemática, Facultad de Ciencias, Universidad de la Rep\'ublica, Igu\'a 4225, 11400 - Montevideo - URUGUAY} 
\email{ivan@cmat.edu.uy}
\address{Alvaro Rittatore, Centro de Matemática, Facultad de Ciencias, Universidad de la Rep\'ublica, Igu\'a 4225, 11400 - Montevideo - URUGUAY} \email{alvaro@cmat.edu.uy}
\maketitle
\section{Introduction}\label{intro}

Let $\Bbbk$ be an algebraically closed field and denote by $\PP^n$ the
projective space of dimension $n$ over $\Bbbk$. The set $\birn$  of
birational maps $f:\PP^n\tor\PP^n$ is the so-called Cremona group of
$\PP^n$. For  an element  $f\in\bir(\PP^n)$ there exist homogeneous
polynomials of the same degree  $f_0,\ldots, f_n\in
k[x_0,\ldots,x_n]$, without nontrivial common factors, such that if
$\x=(x_0:\cdots:x_n)$ is not a common zero of the $f_i$'s, then
$f(\x)=\bigl(f_0(\x):\cdots:f_n(\x)\bigr)$. The (algebraic) degree of $f$ is the
common degree of the $f_i$'s, and is 
denoted by $\deg(f)$.  

A natural way to produce an ``algebraic family'' of  birational maps
is to consider a birational map  $f=(f_0:\cdots:f_n)\in \birn$ and to
allow the coefficients of the $f_i$'s vary in an affine
(irreducible) $\Bbbk$-variety $T$. That is, we consider  polynomials
$f_0,\ldots, f_n\in \Bbbk[T]\otimes \Bbbk[x_0,\ldots,x_n]$, homogeneous and of
the same degree in $\x$ and we define $\varphi:T\to \birn$ by  
\[
\varphi(t,\x)=\bigr(f_0(t,\x):\cdots:f_n(t,\x)\bigl);
\]
in particular we assume that for all $t\in T$ the map $\varphi_t:=\varphi(t,\cdot):\PP^n\tor\PP^n$ is birational.  

As pointed out by Serre in \cite[\S 1.6]{Se} the family of the
topologies  on $\birn$ which make any such algebraic family a 
continuous
function, has a \emph{finest} element, designed in \emph{loc.~cit} as the \emph{Zariski Topology}
of $\birn$. Moreover, we can replace $\PP^n$ with an irreducible
algebraic variety $X$ of dimension $n$ and the same holds for
$\bir(X)$.

The aim of this work is to study the behavior of these ``morphisms''
$T\to\bir(X)$ and obtain, as application, some insight about the
relationship between the
topology and the  algebraic structure of the group $\bir(X)$,  where $X$ is
a rational 
variety.  

More precisely, in Section 2 we present some basic results about $\bir(X)$
that show  the relationship between the algebraic structure and the
Zariski topology.

In Section 3, the main one,  we deal with the case $X=\PP^n$, or
more generally the case where $X$ is a rational variety(see Lemma \ref{lem:functoriality}). We
begin  by stating  two deep
results about the connectedness and simplicity of $\birn$ proved in
\cite{Bla11} and \cite{CaLa} (Proposition \ref{pro3.3}) and extract as
an easy consequence that   a nontrivial normal subgroup
of $\bir(\PP^2)$ 
has trivial centralizer. 
 Next we prove that for a morphism $\varphi:T\to \birn$, the function
$t\mapsto \deg(\varphi_t)$ is lower semicontinuous (\S 3.2). This
result has some nice consequences:

\begin{itemize}

\item[(a)] every Cremona transformation of degree $d$ is a specialization of Cremona transformations of degree $>d$ (Corollary \ref{cor:degeneration});

\item[(b)] the degree map $\deg:\birn\to \Z$ is lower semicontinuous (\S 3.3);
 
\item[(c)] a morphism $T\to\birn$ maps constructible sets into constructible sets (\S 3.5);

\item[(d)] the Zariski topology of $\birn$ is not Noetherian (\S 3.6); 

\item[(e)]  there exist (explicit, non canonical) closed immersions of
  $\bir(\PP^{n-1})\hookrightarrow \birn$ (\S 3.7);

\item[(f)] the subgroup 
  consisting of the elements 
  $f\in\birn$ which stabilize the set of lines passing through a fixed
  point is closed (\S 3.7).   
\end{itemize}

%\noindent{\bf Acknowledgements} 
\section{Generalities}

 Following \cite[\S
2]{Dem} we have:

\begin{defi}\label{defi_pseudo}
A birational map $\varphi:T\times X\tor T\times X$, where $T$
and $X$ are $\Bbbk$-varieties and $X$ is irreducible, is said to be a
\emph{pseudo-automorphism of $T\times X$, over $T$},  if there exists
a dense open subset $U\subset T\times X$ such that:

\begin{itemize}
\item[(a)] $\varphi$ is defined on $U$;

\item[(b)] $U_t:=U\cap \bigl(\{t\}\times X\bigr)$ is dense in $\{t\}\times X$ for all
$t\in T$, and 
     
\item[(c)] there exists a morphism $f:U\to X$ such that
$\varphi|_{_{U}}(t,x)=\bigl(t,f(t,x)\bigr)$, and $\varphi|_{_{U_t}}: U_t\to \{t\}\times
X$ is a birational morphism. 
\end{itemize}
\end{defi}
In particular, a pseudo-automorphism $\varphi$ as above induces a
family $T\to \bir(X)$  of birational maps $\varphi_t:X\tor
X$. Following \cite{Bla11} we call this family an {\em algebraic
  family} in $\bir(X)$ or a {\em morphism} from $T$ to
$\bir(X)$. 

We will identify a morphism $\varphi:T\to\bir(X)$ with its
corresponding pseudo-automorphism and denote $\varphi_t=\varphi(t)$.      

Note that if $\varphi:T\to\bir(X)$ is a morphism, the map $\psi:T\to\bir(X)$ defined by $\psi_t=\varphi_t^{-1}$ is also a morphism where $\varphi_t^{-1}$ denotes the inverse map of $\varphi_t$. 
  
We say $\calf\subset \bir(X)$ is {\em closed} if its pullback under
every morphism $T\to\bir(X)$ is closed in $T$, for all $T$. This
defines the so-called  \emph{Zariski topology} on $\bir(X)$ (\cite{Mum},
\cite[\S 1.6]{Se}, \cite{Bla11}).

In order to define the Zariski topology, as above, it suffices to consider morphisms
from an affine variety $T$. Indeed, notice that a subset $F\subset T$
is closed if and only if there exists a cover by open sets $T=\cup
V_i$, with $V_i$ 
affine, such that $F\cap V_i$ is closed in $V_i$, for all $i$. Then we
may restrict a pseudo-automorphism $\varphi:T\times X\tor T\times X$
to each $V_i\times X$ and obtain a pseudo-automorphism
$\varphi_i:V_i\times X\tor V_i\times X$, for every $i$. The
assertion follows easily from the previous remark. Clearly, we may also suppose $T$ is
irreducible. 

Unless otherwise explicitly stated, in the sequel we always suppose $T$ is affine and irreducible.

% \begin{lem}\label{lem:topologies}
% Let $X$ be an algebraic variety, and endow $\bir(X)$  with a topology
% $\mathcal T$. Let 
% $Z\subset \bir(X)$ be a locally closed set with respect to $\calt$
% such that $\mathcal T$ induces a structure of algebraic variety on $Z$. Then 
% the $\mathcal T$-topology  of $Z$ is finer than the induced Zariski
% topology, that is, if $\calf\subset \bir(X)$ is a Zariski closed subset, then 
% $\calf\cap Z$ is $\mathcal T$-closed in $Z$.
% \end{lem}

% \begin{proof}
% Let $\varphi: Z\times X\to Z\times X$ be given by
% $\varphi(z,x)=\bigl(z,z(x)\bigr)$. Cleary, $\varphi$ is a
% pseudo-automorphism such that $\varphi^{-1}(\calf)=\calf\cap Z$. Since $\calf$
% is closed in $\bir(X)$ for the Zariski topology, it follows that
% $\calf\cap Z$ is closed in $Z$.
% \end{proof}

\begin{lem}
\label{lem:functoriality}
Let $F:X\dashrightarrow Y$ be a birational map between two algebraic
varieties. Then the map $F^*:\bir(Y)\to \bir(X)$ defined by $F^*(f)= F^{-1}\circ f\circ
F$ is a homeomorphism, with inverse $(F^{-1})^*$.
\end{lem}
\begin{proof}
The result follows once we observe that  $ \varphi:T\times Y\dashrightarrow
T\times Y$ is a
pseudo-automorphism if and only if  $(\operatorname{id}\times F^{-1})\circ
\varphi\circ   (\operatorname{id}\times F): T\times X\dashrightarrow
T\times X$ is a pseudo automorphism.  
\end{proof}

We consider $\bir(X)\times \bir(Y)\subset \bir(X\times
Y)$ by taking $(f,g)\in\bir(X)\times \bir(Y)$ into the rational map $F:X\times Y\to X\times Y$ defined as $F(x,y)=\bigl(f(x), g(y)\bigr)$. 

\begin{lem}\label{lem1}
Let $X, Y$ be algebraic varieties and $F\in
\bir(X\times Y)$ a birational map; write $F(x,y)=\bigl(F_1(x,y),F_2(x,y)\bigr)$ for $(x,y)\in X\times Y$ in the domain of $F$. Then $F\in \bir(X)\times \bir(Y)\subset \bir(X\times
Y)$  if and only if there exist dense open subsets $U\subset X$, $V\subset Y$ such that $F$ is defined on $U\times V$ and $F_1(x,y)=F_1(x,y')$, 
$F_2(x,y)=F_1(x',y)$ for $x,x'\in U$, $y,y'\in V$,
\end{lem}
\begin{proof}
First suppose there exist $f\in\bir(X)$ and $g\in\bir(Y)$ such that
$F(x,y)=\bigl(f(x),g(y)\bigr)$. Consider nonempty open sets $U\subset X$
and $V\subset Y$ such that $f$ and $g$ are defined on $U$ and $V$
respectively. Hence, $F_1$ and $F_2$ are defined on $U\times V$ and we
have that $F_1(x,y)=f(x)$ and $F_2(x,y)=g(y)$, from which the ``only if
part'' follows. 

Conversely, suppose there exist nonempty open sets $U$ and $V$ as
stated. Then $F_1$ and $F_2$ induce morphisms $f:U\to X$ and $g:V\to
Y$ such that $F(x,y)=\bigl(f(x),g(y)\bigr)$ for $(x,y)\in U\times V$. Since
$U\times V$ is dense in $X\times Y$, this completes the proof.   
  \end{proof}

\begin{pro}\label{pro1.1}
If  $X, Y$ are algebraic (irreducible) varieties, then 
$\bir(X)\times \bir(Y)\subset \bir(X\times Y)$ is a closed subgroup.
\end{pro}
\begin{proof}
In view of Lemma \ref{lem:functoriality}, we can assume that $X\subset\A^n,Y\subset\A^m$ are
affine  varieties. Let $\varphi:T\times X\times Y\dashrightarrow
T\times X\times Y$ be a pseudo-automorphism (over $T$). Then 
\[\varphi(t,x , y)=\bigl(t,f_1(t,x,y),\dots,
f_n(t,x,y),g_1(t,x,y),\dots, g_m(t,x,y)\bigr),\] 
where $f_i,g_j\in \Bbbk(T\times X\times Y)$ are rational functions on $T\times X\times Y$ (of course, $f_i,g_j$ verify additional conditions). 

Let $A:=\varphi^{-1}\bigl(\bir(X)\times \bir(Y)\bigr)$ and denote by
$\overline{A}$ the closure of $A$ in $T$. Following Lemma \ref{lem1}
it suffices to prove that the restrictions of the $f_i'$s (resp.~the
$g_j'$s) to $\overline{A}\times X\times Y$ do not depend on $y$
(resp.~on $x$), which implies $A=\overline{A}$.  

Up to restrict $\varphi$ to each irreducible component of
$\overline{A}$ we may suppose that $A$ is dense in $T$. By
symmetry we only consider the case relative to the $f_i'$s and write
$f=f_i$ for such a rational function. 

Since the  poles of $f$ are contained in a proper subvariety of
$T\times X\times Y$, we deduce that there exists $y_0\in Y$ such that the
restriction of $f$ to $T\times X\times \{y_0\}$ induces a rational
function on this subvariety. If $p:T\times X\times Y\to T\times
X\times \{y_0\}$ denotes the morphism $(t,x,y)\mapsto (t,x,y_0)$ we
conclude $f\circ p$ is a rational function on  $T\times X\times Y$.  

Our assumption implies $f$ coincides with  $f\circ p$ along $A\times
X\times Y,$ 
 which is dense in $T\times X\times Y$, so $f=f\circ p$  and the
 result follows. 
\end{proof}

\begin{rem}\label{rem5}
Two pseudo-automorphisms
$\varphi:T\times X\tor T\times X$ and $\psi:T\times Y\tor T\times Y$
induce a morphism $(\varphi,\psi): T\to 
\bir(X)\times\bir(Y)$, that is, an algebraic family in 
$\bir(X)\times\bir(Y)$. As in the proof of Proposition \ref{pro1.1},
it follows from Lemma \ref{lem1} that $\mathcal F\subset
\bir(X)\times\bir(Y)$ is closed if and only if
$(\varphi,\psi)^{-1}(\mathcal F)  $  is closed for every pair
$\varphi,\psi$.  
Moreover, it is easy to prove that the topology on $\bir(X)\times\bir(Y)$ induced
by the Zariski topology of $\bir(X\times Y)$ is the \emph{finest} topology
for which all the morphisms  $(\varphi,\psi)$ are
continuous. 

Observe that  the Zariski topology of
$\bir(X)\times\bir(Y)$ is finer than the product topology of the
Zariski topologies of its factors, as it is the case for algebraic varieties. 
\end{rem}

\begin{pro}\label{pro2.1}
If $\varphi,\psi:T\to \bir(X)$ are morphisms, then $t\mapsto
\varphi_t\circ \psi_t$ defines an algebraic family in
$\bir(X)$. Moreover, the product homomorphism $\bir(X)\times
\bir(X)\to \bir(X)$ and the inversion map $\bir(X)\to\bir(X)$ are
continuous. 
\end{pro}
\begin{proof}
To prove the first assertion it suffices to note that the family
$t\mapsto \varphi_t\circ \psi_t$ corresponds to the
pseudo-automorphism $\varphi\circ \psi: T\times X\tor T\times
X$. Applying Remark \ref{rem5},
the first part of the second assertion follows. Indeed, if $\mathcal F\subset
\bir(X)$ is a closed subset, then
$(\varphi,\psi)^{-1}\bigl(m^{-1}(\mathcal F)\bigr)=(\varphi\circ \psi)^{-1}(\mathcal F)$. For  the rest of the
proof it suffices to note that for a family $\psi$ as above  the
map $t\mapsto \psi_t^{-1}$ defines an algebraic family.  
\end{proof}

\begin{lem}\label{lem2.3}
The Zariski topology on $\bir(X)$ is T1. In particular, if
$\varphi,\psi:T\to \bir(X)$ are two morphisms, then the subset
$\bigl\{t\in T; \varphi(t)=\psi(t)\bigr\}$ is closed. 
\end{lem}

\begin{proof}
It suffices to show that $id\in\bir(X)$ is a closed point. Without
loss of generality we may suppose $X\subset \PP^m$ is a projective
variety. Then a morphism $\varphi:T\to\bir(X)$ may be
represented as 
\[
\varphi_t=\bigl(f_0(t,x):\cdots:f_m(t,x)\bigr), (t,x)\in U,
\]
where $U$ is as in Definition \ref{defi_pseudo} and $f_i\in \Bbbk[T][x_0,\ldots,x_m]$, $i=0,\ldots,m$, are
homogeneous of same degree in the variables
$x_0,\ldots,x_m$. Therefore 
\[
\begin{split}
\bigl\{t\in T; \varphi(t)=id\bigr\} & =
\bigcap_{i,j=0}^m\bigl\{t\in T \mathrel{:} x_jf_i(t,x)-x_if_j(t,x)=0,
\ \forall (t,x)\in U_t\bigr\}\\
& =
\bigcap_{i,j=0}^m\bigl\{t\in T \mathrel{:} x_jf_i(t,x)-x_if_j(t,x)=0,
\ \forall x\in X\bigr\}\\
& = \bigcap_{i,j=0, x\in X}^m\bigl\{t\in T \mathrel{:}
x_jf_i(t,x)-x_if_j(t,x)=0 \bigr\}.
\end{split}
\] 

Since  for all $i,j$ the equations 
\[x_jf_i(t,x)-x_if_j(t,x)=h_1(x)=\cdots=h_\ell(x)=0\]
define a closed set in $T\times X$, and $X$ is projective we deduce
$\bigl\{t\in T \mathrel{:} \varphi(t)=id\bigr\}$ is closed in $T$. 
\end{proof}

\begin{cor}
Let $\psi:Y\to \bir(X)$ be a morphism, where $Y$ is a projective variety. Then $\psi(Y)$ is closed. 
\end{cor}
\begin{proof}
A morphism $\varphi:T\to \bir(X)$ induces a morphism $\phi: T\times
Y\to \bir(X)$ defined by $(t,y)\mapsto \varphi(t)\circ\psi(y)^{-1}.$
Then $\phi^{-1}\bigl(\{id\}\bigr)=\bigl\{(t,y); \varphi(t)=\psi(y)\bigr\}$ is closed in
$T\times Y$. The projection of this set onto the first factor is
exactly $\varphi^{-1}\bigl(\psi(Y)\bigr)$ which is closed.     
\end{proof}

\begin{cor}\label{cor2.6}
The centralizer of an element $f\in\bir(X)$ is closed. In particular,
the centralizer $C_{\bir(X)}(G)$ of a subgroup $G\subset \bir(X)$ is
closed. \hfill $\qed$
\end{cor}
\begin{proof}
Since the commutator map $c_f:\bir(X)\to \bir(X)$,
$c_f(h)=hfh^{-1}f^{-1}$, is continuous, $c_f^{-1}\bigl(\{id\}\bigr)$ is closed. 
\end{proof}

Another consequence of Lemma \ref{lem2.3} (and Remark \ref{rem5}) is
that for an arbitrary topological subspace $A\subset\bir(X)$ and a
point $f\in\bir(X)$, the natural identification map $\{f\}\times A\to
A$ is an homeomorphism. As in \cite[Chap.I, Thm. 3]{Sha} we obtain: 

\begin{cor}\label{cor2.7}
If $A,B\subset\bir(X)$ are irreducible subspaces, then $A\times B$ is
an irreducible subspace of $\bir(X)\times\bir(X)$.
\end{cor}

\begin{pro}\label{pro2.8}
The irreducible components of $\bir(X)$ do not intersect. Moreover,
$\bir(X)^0$, 
the  unique irreducible component of $\bir(X)$ which contains
$id$,  is a
normal (closed) subgroup.   
\end{pro}

\begin{proof}
Let $A,B$ be  irreducible components  containing $id$.  Corollary
\ref{cor2.7} implies $A\cdot B$ is irreducible. Since $id\in A\cap
B$ then $A \cup B \subset A\cdot B$ from which it follows $A=A\cdot B=B$.  This proves the
uniqueness of $\bir(X)^0$.  

The rest of the proof works as in \cite[Chapter 3, Thm. 3.8]{FR}.     
\end{proof}

We have also the following easy result:

\begin{pro}\label{pro2.9}
Let $H\subset \bir(X)$ be a subgroup. 

(a) The closure $\overline{H}$ of $H$ is a subgroup. Moreover, if $H$ is normal, then $\overline{H}$ is normal. 

(b) If $H$ contains a dense open set, then $H=\overline{H}$.
\end{pro}

\begin{proof}
The proof of this result follows the same arguments that the analogous
case for algebraic groups (see  \cite[Chapter
3, Section 3]{FR}). For example, in order to prove the second part of (a) it
suffices to note that since $g\mapsto fgf^{-1}$ is an homeomorphism, then
$f\overline{H}f^{-1}=\overline{fHf^{-1}}$. 
\end{proof}

% We would like to have a version of Chevalley's results on constructibility as follows:

% \begin{con}\label{con2.1}
% Let $\varphi:T\to\bir(X)$ be a morphism and let $C\subset T$ be a constructible set. Then $\varphi(C)$ is constructible and contains a dense open set of $\overline{\varphi(C)}$. 
% \end{con}  

% The conjecture, if true, motivates the following:

% \begin{defi}
% A closed set $\calf\subset\bir(X)$ is said to be a (closed) subvariety if there exists a morphism $f:T\to\bir(X)$ such that $\calf=\overline{f(T)}$. 
% \end{defi}

% In Theorem \ref{thm:chevalley} we prove the conjecture in the case where $X$ is a rational variety.

\section{The Cremona group}\label{sec3}

In this section we consider the case $X=\PP^n$; we  fix homogeneous coordinates
$x_0,\ldots,x_n$ in $\PP^n$. As in the introduction, 
if $f:\PP^n\tor\PP^n$ is a birational 
map, the \emph{degree} of $f$ is the minimal degree $\deg(f)$ of
homogeneous polynomials in $\Bbbk[x_0,\ldots,x_n]$ defining $f$.  

\subsection{Connectedness and simplicity}\ %

In \cite[Thms. 4.2 and 5.1]{Bla11} J\'er\'emy Blanc proves the following two results: 

\begin{thm}[J. Blanc]\label{bla1}
$\birplane$ does not admit nontrivial normal closed subgroups.
\end{thm} 

\begin{thm}[J. Blanc]\label{bla2}
If $f,g\in \birn$, then there exists a morphism $\theta:U\to\birn$, where $U$ is an open subset of $\A^1$ containing $0,1$, such that $\theta(0)=f, \theta(1)=g$. In particular $\birn$ is connected.    
\end{thm} 

In Theorem \ref{bla2} the open set $U$ is irreducible and the morphism
$\theta$ is continuous. Hence we deduce that $\birn$ is irreducible.

On the other hand, in \cite{CaLa} Serge Cantat and St\'ephane Lamy prove the following result:

\begin{thm}[S. Cantat-S. Lamy]\label{cala}
$\birtwo$ is not a simple (abstract) group, i.e., it contains a non trivial normal subgroup. 
\end{thm}

In fact they prove that for a ``very general'' birational map
$f\in\birtwo$ of degree $d$, with $d\gg 0$, the minimal normal
subgroup containing $f$ is nontrivial. From Theorems \ref{bla1} and
\ref{cala} it follows that all non trivial normal subgroup in
$\birtwo$ are dense. 

Putting all together we obtain:

\begin{pro}\label{pro3.3}
Let $G\subset \bir(\PP^2)$ be a nontrivial normal subgroup. Then  $C_{\bir(\PP^2)}(G)=\{id\}$.  
\end{pro} 
\begin{proof}
Suppose $C_{\bir(\PP^2)}(G)\neq \{id\}$. The closure $\overline{G}$ of
$G$ is a normal subgroup, then it coincides with the entire Cremona
group. If $f\in C_{\bir(\PP^2)}(G)$, then $G$ is contained in
the centralizer of $f$, which is closed. We deduce that $f$ commute
with all the elements of $\bir(\PP^2)$, that is $C_{\bir(\PP^2)}(G)$
coincides with the center $Z(\bir(\PP^2))$ of ${\bir(\PP^2)}$. Since 
  $Z(\bir(\PP^2))=\{id\}$, the result follows.   
 For the convenience of the reader we give a
 proof of the well known fact that  $Z(\bir(\PP^2))=\{id\}$.

Recall that $\bir(\PP^2)$ is generated by quadratic transformations of the form $g_1\sigma g_2 $ where
$g_1,g_2\in\pgl(3,\Bbbk)$ and $\sigma=(x_1x_2:x_0x_2:x_0x_1)$ is the
\emph{standard} quadratic transformation. Take $f\in
Z(\bir(\PP^2))$. If $L\subset\PP^2$ is a general line, then we may
construct a quadratic transformation $\sigma_L$ which contracts $L$ to
a point and such that $f$ is well defined in this point. Since
$f\sigma_L=\sigma_L f$ and we may suppose $f$ is well defined and
injective on an open set of $L$ we deduce $f$ transforms $L$ into a
curve contracted by $\sigma_L$, that is, the strict transform of $L$
under $f$ is a line, and then $f\in\pgl(3,\Bbbk)$, so $f\in
Z(\pgl(3,\Bbbk))=\{id\}$.   
  \end{proof}

\subsection{Writings and degree of a pseudo-morphism}\ %

Let $\varphi:T\to\bir(\PP^n)$ be a morphism, where $T$ is an 
irreducible variety. Denote by $\pi: T\times\PP^n\to T$ the projection
onto the first factor. Then the pseudo-automorphism $\varphi$ (Definition
\ref{defi_pseudo})  verifies the following commutative diagram   
\[
\xymatrix{T\times \PP^n\ar@{-->}[rr]^\varphi\ar@{->}[rd]_\pi& &T\times \PP^n\ar@{->}[ld]^\pi\\
&T&}
\]
In other words, $\varphi$ induces  a commutative diagram 
\[
\xymatrix{\Bbbk(T\times \PP^n)& &\Bbbk(T\times \PP^n)\ar@{-->}[ll]_{\varphi^*}\\
& \Bbbk(T)\ar@{^{(}->}[ru]_{\pi^*}\ar@{_{(}->}[lu]^{\pi^*}& }.
\]

We deduce that there exist rational functions
$\varphi_0,\ldots,\varphi_n\in\Bbbk(T\times \PP^n)$ such that  
\[
\varphi(t,\x)=\bigl(\varphi_0(t,\x):\cdots:\varphi_n(t,\x)\bigr),
\]
where the formula above holds for $(t,\x)$ in an open set $U\subset 
T\times \PP^n$. Moreover,  we may suppose $U\cap (\{t\}\times
\PP^n)\neq \emptyset$ for all $t\in T$. Observe that we are assuming
that  $U$ is contained in the 
domain of definition  of $\varphi_i$, for all $i$. Hence, for all $t\in
T$, there exists an  open
set  $U_{t}\subset \{t\}\times \PP^n$ where all $\varphi_i|_{{U_t}}$ are well
defined. We can also assume that there exists $i_t$ such that
$\varphi_{i_t}$ does not vanish in $U_{t}$. 

Let $V\subset T$ be an affine nonempty open subset. From the remarks
above, we deduce that 
there exists a (non necessarily unique) representation of  $\varphi$ of the form  
\begin{equation}
\varphi(t,\x)=\bigl(f_0(t,\x):\cdots:f_n(t,\x)\bigr), (t,\x)\in
U'\subset U\cap(V\times \PP^n),
\label{writing}
\end{equation}
where $U'\subset U\cap(V\times \PP^n) $ is an open subset and $f_0,\ldots,f_n\in \Bbbk[V\times\A^{n+1}]=\Bbbk[V]\otimes
\Bbbk[x_0,\ldots,x_n]$ are homogeneous polynomials  in
$x_0,\ldots,x_n$, of the same 
degree. In particular, if $U'\cap
\bigl(\{t_0\}\times \PP^n\bigr)\neq \emptyset$, then 
\[
\varphi_{t_0}(\x)= \bigl(f_0(t_0,\x):\cdots:f_n(t_0,\x)\bigr)
\]
for $\x$ in an open set $U'_{t_0}\subset \PP^n$;  that is, there exist
$\x_0\in U'_{t_0}$ and $i_0$ such that  $ f_{i_0}(t_0,\x_0)\neq 0$. Observe that $
\{t_0\}\times U'_{t_0}\subset U_{t_0}$. 

\begin{defi}\label{defi:writing}
With the notations above, consider the  $(n+1)$-uple
$(f_0,\ldots,f_n)$ satisfying  (\ref{writing}) and let $\ell=\deg(f_i)$. We say that 
$w_V^\varphi=(f_0,\dots,f_n)$ is a \emph{writing of  $\varphi$ on
  $V$}. The positive integer $\deg(w_V^\varphi):=\ell$ is said to be the
\emph{degree} of $w_V^\varphi$. 
\end{defi}

% \begin{rem}
% If $V\subset T$ is an affine open set, then the morphism  $\varphi$
% restricts to $V$ and defines a morphism $\varphi_V: V\to
% \bir(\PP^n)$. The above construction gives a writing
% $w_V^{\varphi_V}=(g_0,\dots,g_n)$ for $\varphi_V$ and clearly 
% \[\varphi(t,\x)=\bigl(g_0(t,\x):\cdots:g_n(t,\x)\bigr)\]
% for $(t,\x)$ in an open set of $V\times \PP^n$. We denote $w_V^\varphi =w_V^{\varphi_V}$ and say that $w_V^\varphi$ is a \emph{writing of $\varphi$ on $V$}. 
% \end{rem}

\begin{rem}
Let $ w=w^\varphi_V=(f_0,\dots,f_n)$ be a writing of $\varphi$ on an
affine open subset $V\subset T$. We introduce the ideal $I(w)\subset
\Bbbk[V]\otimes \Bbbk[\x]$ 
generated by $f_0, \ldots, f_n$. Then  $I(w)$ defines a subvariety
$X^w\subset V\times \A^{n+1}$. Notice that $X^w$ is stable
under the action of $\Bbbk^*$ on $V\times \A^{n+1}$ defined by
$\lambda\cdot (t,x)\mapsto (t,\lambda x)$. Moreover, the projection
$\pi:X^w\to V$ onto the first factor is equivariant and, by
definition, surjective.  The function $t\mapsto \dim\pi^{-1}(t)$ is
upper-semicontinuous, from which we deduce $V_{i}:=\{t;
\dim\pi^{-1}(t)\geq i\}$ is closed in $V$ for all $i=1, \dots, n+1$. 

Since
$\pi^{-1}(t)=X^{w}\cap \bigl((\{t\}\times\A^{n+1}\bigr)$, it follows that
$\dim\pi^{-1}(t)>n$ if and only if $\pi^{-1}(t)=\{t\}\times \A^{n+1}$. In
other words, an element $t\in V$ belongs to $V_{n+1}$ if and only if
$\bigl(\{t\}\times \PP^n\bigr)\cap U'=\emptyset$, where $U'\subset
V\times \PP^n$ is the domain of definition of the rational map $(t,\x)
\mapsto \bigl(t, \bigl(f_0(t,\x):\dots:
f_n(t,\x)\bigr)\bigr)$. Observe that $V_{n+1}\subsetneq V$. 
 \end{rem}

The preceding remark motivates the following

\begin{defi}
Let $\varphi:T\to \birn$ be a morphism and $t\in T$. A \emph{writing
  passing through $t$} is a writing $w^\varphi_V$ of $\varphi$ such
that $t\in V\setminus V_{n+1}$.
\end{defi}

\begin{lem}
Let $\varphi:T\to \birn$ be a morphism and $t_0\in T$. Then there exists
a writing $w^\varphi_V$ passing through $t_0$.
\end{lem}
\begin{proof}
By definition, there exists $\x_0\in \P^n$ such that $\varphi$ is
defined in $(t_0, \x_0)$. Hence, there exist $f_0,g_0,\dots, f_n,g_n\in
\Bbbk[T]\otimes \Bbbk[\x]$ such that $(g_0\cdots g_n)(t_0,\x_0)\neq 0$
and $\varphi(t,  \x)=\bigl(f_0/g_0(t,\x): \dots : f_n/g_n(t,\x)\bigr)$,
where the equality holds in an open neighborhood $A$ of $(t_0,\x_0)$ in $ T\times
\P^n$. Eliminating  denominators, we deduce  that
\[
\varphi(t,  \x)=\bigl(h_0(t,\x): \dots : h_n(t,\x)\bigr),
\]
where $h_i\in \Bbbk[T]\otimes \Bbbk[\x]$ and the above formula holds in an open subset $A' \subset A$,
containing $(t_0,\x_0)$. If $V\subset T$ is an affine open subset such that for all $t\in
V$ there exists $\x\in\P^n$ with $(t,\x)\in A'$, it is clear that
$w^\varphi_V=(h_0,\dots,h_n) $ is a writing of 
$\varphi$ through $t_0$.
\end{proof}

\begin{defi}
Let $\varphi:T\to\bir(\PP^n)$ be a morphism, where $T$ is an irreducible
variety. Denote by $\calv$ the family of nonempty affine open sets in
$T$ on which there exists, at most, a writing of $\varphi$. The  \emph{degree} of $\varphi$ is the positive integer 
\[
\Deg(\varphi):=\min\{\deg(w_V^\varphi): V \in\calv \}.
\]
\end{defi}

Note that two $n$-uples $(f_0,\ldots,f_n)$ and $ (f'_0,\ldots,f'_n)$,
with $\deg(f_i)=\deg(f'_i)=\Deg(\varphi)$ define the same writing on
an open set $V$ if and only if they  
coincide up to multiplication by a nonzero element in
$\Bbbk(V)=\Bbbk(T)$.

For $t\in T$ we denote by $\deg(\varphi_t)$ the usual
algebraic degree of the map $\varphi_t:\PP^n\tor\PP^n$; it is the
minimal degree of components among the $(n+1)$-uples of homogeneous polynomials defining
$\varphi_t$. 

By applying (\ref{writing})  we obtain that if $t\in T$, then
$\deg(\varphi_{t})\leq \deg(w_V^\varphi)$ for every  writing
$w_V^\varphi$ passing through $t$. 
 Moreover, we have the following

\begin{lem}\label{lem4.1}
 Let  $w=w_V^\varphi$ be a writing for the morphism $\varphi:T\to
 \birn$ and $t\in V\setminus V_{n+1}$. Then the following assertions are equivalent:
 
(a) $t\in V_{n}$.

(b)  There is a codimension 1 subvariety $X^w_t\subset\A^{n+1}$ such that $\pi^{-1}(t)=\{t\}\times X^w_t$. 

(c) $\deg(\varphi_t)<\deg(w)$.
\end{lem}
\begin{proof}
The  equivalence of assertions (a) and (b) is obvious. In order to
prove that (b) is equivalent to (c)  let  $w=(f_0,
\ldots,f_n)$; then for every $t\in V\setminus V_{n+1}$ the rational
map  
\[
\x\mapsto \bigl(f_0(t,\x): \ldots:f_n(t,\x)\bigr)
\]
coincides with $\varphi_t$. Therefore $\deg(\varphi_t)<\deg(f_i)$ if and only if the polynomials $g_0,\ldots,g_n\in \Bbbk[\x]$ defined by $g_i(\x)=f_i(t,\x)$, where $t$ is fixed and $i=0,\ldots,n$, admit a nontrivial factor.    
\end{proof}

The following example is taken from \cite[Lemma 2.13]{BlFu}

\begin{exa}
Let $T\subset \PP^2$ be the projective nodal cubic curve of equation $a^3+b^3-abc=0$, with singular point $o=(0:0:1)$, and consider the morphism  $\varphi:T\to \birn$ defined by 
\[\varphi(a:b:c)=(x_0f:x_1g:x_2f:\cdots:x_n f),\]
where 
\[f=bx_0^2+cx_0x_2+ax_2^2,\ g=(a+b)x_0^2+(b+c)x_0x_2+ax_2^2;\]
note that $\varphi_o=(x_0^2x_2:x_0x_1x_2:x_0^2x_2:\cdots:x_0^2x_n)$ is the identity map. 

Set $ f'=abf$ and $ g'=abg$, that is
\[f'=ab^2x_0^2+(a^3+b^3)x_0x_2+a^2bx_2^2,\ g'=ab(a+b)x_0^2+(ab^2+a^3+b^3)x_0x_2+a^2bx_2^2.\]
If $V\subset T$ is the affine open set defined by $c=1$, 
then $w^\varphi_V=(x_0f',x_1g',x_2f',\ldots,x_n f')$ is a writing of
$\varphi$ on $V$ with degree 3. Clearly $o\in V_{n+1}$ and
$w^\varphi_V$ is through all non-singular point in $T$. As it follows
from  \emph{loc. cit} the polynomial  $ax_0+bx_2$ defines  (locally) a higher
common divisor for $f'$ and $g'$ in $\Bbbk[V']\otimes \Bbbk[\x]$ where
$V':=V\setminus\{o\}$.  Hence $V=V_n$. Dividing all components in
$w^\varphi_V$ by   $ax_0+bx_2$ we obtain a new  writing on $V$ of  degree $2$. One deduces $\Deg(\varphi)=2$.      
\end{exa}

\begin{rem}\label{rem:DEGDESIGN}
Consider a morphism  $\varphi:T\to
\bir(\PP^n)$ and let $U$ and $f:U\to 
\PP^n$  be as in Definition \ref{defi_pseudo}. If 
 $\sigma:S\to T$ is a birational morphism it follows that the morphism
\[
 (\sigma\times id)^{-1}(U)\to S\times\PP^n, (s,\x)\mapsto \bigl(s,f(\sigma(s),\x)\bigr),
\]
induces a  morphism $\varphi\circ \sigma:S\to \birn$.

If $s\in S$, then  $\bigl((\sigma\times
id)^{-1}(U)\bigr)_s\simeq U_{\sigma(s)}$ and up to this isomorphism the birational map $\varphi_{_{\sigma(s)}}$ coincides
with $(\varphi\circ\sigma)_s$.   
\end{rem}

\begin{lem}\label{lem:DEGDESIGN}
Let  $\varphi:T\to
\bir(\PP^n)$ be a morphism and 
 consider a birational morphism $\sigma:S\to T$. Then
 $\Deg(\varphi)=\Deg(\varphi\circ \sigma)$. 
\end{lem}

\begin{proof}
Notice that if $\varphi:T\to \birn$ is a morphism and $U \subset T$ is an open
subset, then $\varphi|_U:U\to \birn$ is also a morphism, and  
clearly  $\Deg (\varphi)= \Deg (\varphi|_{U})$. Then it suffices to prove the result when  $\sigma$
is an  
isomorphism, in which case the result is trivial.
\end{proof}

% Furthermore, if $w^\varphi_V=(f_0,\ldots,f_n)$ is a writing for
% $\varphi$ on $V$ and $W\subset \sigma^{-1}(V)$ is an affine open
% set, then $f_i\circ(\sigma\times id)|_{W}\in \Bbbk[W]\otimes
% \Bbbk[\x]$, $i=0,\ldots,n$ and $(g_0,\ldots,g_n)$ defines  a writing
% $w^{\varphi\circ\sigma}_W$ for  $\varphi\circ\sigma$ on $W$; we say
% this is the writing of $\varphi\circ\sigma$ on $W$ induced from
% $w^\varphi_V$ by $\sigma$. Clearly
% $\deg(w^\varphi_V)=\deg(w^{\varphi\circ\sigma}_W)$. If in addition
% $w^\varphi_V$ passes through a point $t\in T$ and $\sigma(s)=t$, for
% $s\in S$, then $w^{\varphi\circ\sigma}_W$ passes through $s$; notice
% that the converse does not hold as shows the following example. 

\subsection{Degree and semicontinuity} \ %

\begin{pro}\label{pro4.1}
Let  $\varphi:T\to\bir(\PP^n)$ be a morphism. Consider the set 
$U_\varphi:=\{t\in T \mathrel{:} \deg(\varphi_t)=\Deg(\varphi) \}$. Then

(a) $U_\varphi$ is a nonempty open subset of $T$.

(b) $\deg(\varphi_{t})\leq \Deg(\varphi)$ for all $t\in T$. 
\end{pro}

\begin{proof}
We may reduce the proof to the case where $T$ is smooth. Indeed, if  $T$ is singular we consider a proper birational surjective morphism $\sigma:S\to T$, where $S$
is smooth, and set $\psi:=\varphi\circ\sigma$; assume that assertions
(a) and (b) hold on $S$. Then  Remark \ref{rem:DEGDESIGN} implies (b)
holds on $T$  and that remark together with Lemma \ref{lem:DEGDESIGN}
imply $U_{\psi}=\sigma^{-1}(U_{\varphi})$. Since $\sigma$ is proper
and surjective it follows that $\sigma$ is an open morphism; hence
$\sigma(U_{\psi})=U_{\varphi}$ is a nonempty open subset of $T$ which
proves that (a) also holds on $T$.  

Now assume $T$ is smooth.   
In order to prove that $U_\varphi$ is not empty we consider a writing
$w_V^\varphi$ such that $\deg 
(w_V^\varphi)=\Deg(\varphi)$. By  Lemma \ref{lem4.1}, it suffices to
prove that  $V\backslash V_n\neq\emptyset$. Assume that $V_{n}=V$ and 
consider  the 
variety $X^w\subset V\times 
\A^{n+1}$ defined by the ideal $I(w)$ generated by the components of $w$. Since
$V_{n+1}\subsetneq V$, it follows that $X^w$ has
codimension 1; denote by $Z$ the union of codimension 1 irreducible
components of $X^w$ which project onto $V$. If $t_0\in V\setminus V_{n+1}$, then the ideal
$I(Z)_{t_0}\subset \calo_{V,t_0}[\x]=\calo_{T,t_0}[\x]$ of elements in
$\calo_{V,t_0}[\x]$ vanishing in a neighborhood of $(\{t_0\}\times
\A^{n+1})\cap Z$ is principal; let $g\in \calo_{V,t_0}[\x]$ be a
polynomial,  homogeneous  in 
$x_0,\ldots,x_n$, which generates $I(Z)_{t_0}$. Hence there exist a
positive integer $\ell$, an index $0\leq j\leq n$ and homogeneous
polynomials  $h_0,\ldots,h_n \in \calo_{V,t_0}[\x]$ such that
$f_i=g^{\ell}h_i$, for all $i=0,\ldots,n$, and $h_j\not\in
I(Z)_{t_0}$.  

There exists an affine open neighborhood  $V'$ of $t_0$ in $V\setminus V_{n+1}$ such
that $f_i,g,h_i\in \Bbbk[V']\otimes\Bbbk[\x]$. Then
$w^{\varphi}_{V'}:=(h_0,\ldots,h_n)$ defines a writing of $\varphi$ on
$V'$ through $t_0$, with $\deg(w^{\varphi}_{V'})< \deg (w^\varphi_V)=\Deg(\varphi)$, and we
obtain a contradiction.

In order to prove that $U_\varphi$ is open, let $t_0\in U_\varphi$ and
consider a writing $w'=w_U^\varphi=(f_0',\dots, f_n')$ passing through
$t_0$. 

If $U\setminus U_n\neq \emptyset$ then $A= (V\setminus V_n)\cap
(U\setminus U_n)\neq \emptyset$ and it follows from Lemma \ref{lem4.1}
that  for all $t\in
A$
\[
\deg(w')=\deg(\varphi_t)=\deg(w)=\Deg(\varphi)=\deg(\varphi_{t_0}).
\]
Hence $t_0\in U\setminus U_n\subset U_\varphi$. 

If  $U=U_n$, by arguing
as in the preceding part of the proof we deduce the existence of an
affine open neighborhood 
$U'\subset U\setminus U_{n+1}$ of $t_0$ and a writing $w^\varphi_{U'}=(h_0',\dots , h_n')$,
with $f_i={g'}^{\ell'}h'_i$ for some $g', h_i'\in \Bbbk[U']\otimes
\Bbbk[\x]$. Since  $h'_j$ does not belong to $I(Z')_{t_0} $ (obvious notations), Lemma
\ref{lem4.1}(c) implies
$\deg(w^{\varphi}_{U'})\leq\deg(\varphi_{t_0})$, and thus
$\deg(w^{\varphi}_{U'})=\Deg(\varphi)$. Hence $t_0\in U'\setminus
U'_{n}\subset U_\varphi$ which completes the proof of $(a)$.

In order to prove that $U_\varphi$ is open, let $t_0\in U_\varphi$ and
consider a writing $w_U^\varphi=(f_0',,\dots f_n')$ passing through
$t_0$. If $t_0\in U\setminus U_n$ there is noting to prove. Otherwise
$\deg(w_U^\varphi)>\deg(\varphi_{t_0})=\Deg(\varphi)$, hence $U=U_n$. By arguing as in the preceding part of the proof we deduce the existence of an affine open neighborhood
$U'\subset U\setminus U_{n+1}$ of $t_0$ and a writing $w^\varphi_{U'}=(h_0',\dots , h_n')$,
with $f_i={g'}^{\ell'}h'_i$ for some $g', h_i'\in \Bbbk[U']\otimes
\Bbbk[\x]$. Since  $h'_j$ does not belong to $I(Z')_{t_0} $ (obvious notations), Lemma
\ref{lem4.1}(c) implies
$\deg(w^{\varphi}_{U'})\leq\deg(\varphi_{t_0})$, and thus
$\deg(w^{\varphi}_{U'})=\Deg(\varphi)$. Hence $t_0\in U'\setminus
U'_{n}\subset U_\varphi$ which completes the proof of $(a)$.

To prove (b) we consider a writing $w=w_V^\varphi=(g_0,\dots,g_n)$
such that 
$\deg(w)=\Deg(\varphi)$. Since $g_i\in\Bbbk[V]\otimes \Bbbk[\x]\subset
\Bbbk(T)[\x]$ for all $i$,  there 
exists $a\in\Bbbk[T]$ such that  $ag_i\in\Bbbk[T]\otimes \Bbbk[\x]$ for
$i=1,\ldots,n$. Write 
\[
ag_i=\sum_{I\in\cali} a^i_I {\x}^I,\ \cali=\bigl\{I=(i_0,\ldots,i_n);
i_0+\cdots+i_n=\Deg(\varphi)\bigr\}, a^i_I\in\Bbbk[T],
\]
for $i=0,\ldots,n$.
  
If $t\in T$ we take an irreducible smooth curve $C\subset T$ 
passing through $t$ such that $C\cap U_\varphi\neq
\emptyset$. If $\alpha$ is a local parameter for the local ring
$\calo_{C,t}$ of $C$ at $t$, there exists $m$ such that $\alpha^m$
does divide the restriction of $a^i_I$ to $C$, for all $I$ and all $i$, but  $\alpha^{m+1}$ does
not; set 
\[g'_i:=\sum_{I\in\cali} b^i_I {\x}^I,\]
where $b_I^i:=(a^i_I|_{C})/\alpha^m\in\calo_{C,t}$, $i=0,\ldots,n$. By construction $(g'_0,\ldots,g'_n)$ defines a writing of the morphism $\varphi|_{C}:C\to\bir(\PP^n)$ on an open neighborhood  of $t$ in $C$. It follows $\deg(\varphi_{t})\leq \deg(g'_i)=\deg(g_i)=\Deg(\varphi)$.   
\end{proof}

As a consequence of (the proof of) Proposition \ref{pro4.1} we have
the following:

\begin{cor}\label{corowritngdeg}
Let $\varphi:T\to \birn$ be  a morphism, then:

\noindent $(a)$ $\Deg(\varphi)=\max \bigl\{
\deg(\varphi_t)\mathrel{:} t\in T\bigr\}$. Moreover, a
writing $w_V^\varphi$ is of minimum degree,  that is  
$\deg(w^\varphi_V)=\Deg(\varphi)$, if and only if $V\setminus V_n\neq
\emptyset$.\hfill 

\noindent $(b)$ If $t\in T$ is such that $\deg(\varphi_t)=\Deg (\varphi
)$, then there exists a writing $w=w^\varphi_V$ through $t$, with
$\deg(w)=\Deg(\varphi)$.   
\qed  
\end{cor}

Clearly the function $t\mapsto\deg(\varphi_t)$ takes finitely many
values, say $d_1=\Deg(\varphi)>d_2>\cdots>d_\ell\geq 1$. Consider the
decomposition $T\backslash U_\varphi=X_1\cup\cdots\cup X_r$ in
irreducible components. We may restrict $\varphi$ to each $X_i$  and
apply Proposition \ref{pro4.1} to conclude $\deg(\varphi_t)=d_2$ for
$t$ in an open set (possibly empty for some $i$) $U_i\subset X_i$ and $\deg(\varphi_t)<d_2$ on $X_i\backslash
U_i$, $i=1,\ldots,r$. Repeating  the argument with $d_3$, and so
on, we deduce: 

\begin{thm}\label{thm4.2}
Let $\varphi:T\to \bir(\PP^n)$ be a morphism.   Then 

(a) There exists a stratification by locally closed sets $T=\cup_{j=1}^\ell V_j$ such that $\deg(\varphi_t)$ is constant on $V_j$, for all $j=1,\ldots,\ell$. 

(b) The function $\deg{\scriptstyle \circ}\varphi:T\to \N$,  $t\mapsto
\deg(\varphi_t)$, is lower-semicontinuous.   \hfill \qed
\end{thm}

\begin{cor}\label{cor:degeneration}
If $d, e\in\Z$ are positive integers numbers with $d\leq e$, then every Cremona transformation of degree $d$ is specialization of
Cremona transformations of degrees $\geq e$.  
\end{cor}
\begin{proof}
Let $f$ be a Cremona transformation of degree $d$. Consider a morphism
$\theta:T\to \bir(\PP^n)$, where $T$ is a dense open set in $\A^1$
containing $0,1$ such that $\theta(0)=f$ and $\theta(1)$ is a Cremona transformation of degree $e$ (Theorem \ref{bla2}). The proof follows from Proposition \ref{pro4.1} applied to the morphism $\theta$.     
\end{proof}

\begin{cor}\label{cor4.3}
The degree function $\deg:\bir(\PP^n)\to\N$ is  lower-semicontinuous,
i.e. for all $d$ the subset $\bir_{\leq d}(\PP^n)$ of birational maps
of degree $\leq d$ is closed. In particular, a subset
$\calf\subset\bir(\PP^n)$ is closed if and only if  $\calf\cap
\bir(\P^n)_{\leq d}$ is closed for all $d>0$. 
\end{cor}

\begin{proof}
The
assertion relative to semicontinuity is a direct consequence of
Theorem \ref{thm4.2}(b). For the last assertion we note
that if $\varphi:T\to\bir(\PP^n)$ is a morphism and  $e=\Deg(\varphi)$, then
$\varphi^{-1}(\calf)=\varphi^{-1}\bigl(\calf\cap \bir(\P^n)_{\leq e}\bigr)$.  
\end{proof}

\begin{rem}\label{rem4.4}
Note that  $\bir(\PP^n)=\bigcup_{d\geq 1} \bir(\PP^n)_{\leq d}$, with $\bir(\PP^n)_{\leq d}\subsetneq \bir(\PP^n)_{\leq d+1}$ and $\bir(\PP^n)_1=\pgl(n+1,\Bbbk)$.
\end{rem}

% \begin{cor}
% A subset $\calf\subset\cren$ is closed if and only if $\calf\cap\cren_{\leq d}$ is closed for all $d>0$.
% \end{cor}

% \begin{proof}
% For a morphism $\varphi:T\to\bir(\PP^n)$ with $d=\Deg(\varphi)$ we have $\varphi^{-1}(\calf)=\calf\cap\bir(\PP^n)_{\leq d}$.  
% \end{proof}

\subsection{Algebraization of morphisms}\ %

In this paragraph we deal with the morphisms $\varphi:T\to \birn$ and
their relationship with the stratification described in Theorem
\ref{thm4.2}. We consider the locally closed sets
$\birn_d:=\birn_{\leq d}\backslash \birn_{\leq d-1}$, where $d\geq
2$. If  $\Deg(\varphi)=d$, then
$U_\varphi=\varphi^{-1}(\birn_{d})$.  

Nguyen has shown  in his doctoral thesis (\cite{Ngu}) that $\birn_d$
(with the induced Zariski topology)
supports  a structure of algebraic variety
 (see also  \cite[Prop.2.15]{BlFu}). We give here some details on this
 construction, as a preliminary result for Theorem \ref{thm:chevalley}.

 For integers $d,n,r$, with $d,n>0$ and $r\geq 0$, we consider the
 vector space $V=\Bbbk[x_0,\ldots, x_n]_d^{r+1}$ of $(r+1)$-uples of
 $d$-forms.  Notice that the projective space $\P_{(d,n,r)}=\P(V)$ consisting of
 dimension 1 subspaces in $V$ has dimension $N(d,n,r)={n+d\choose
   d}(r+1)-1$.  

% For $r=0$ and $d=e$ this construction gives a projective space
% $\P^m$ with $m={n+e\choose e}-1$ and replacing $d$ with $d-e$, when
% $e<d$, it gives a projective space $\P^M$ where $M=N(d-e,
% n)={n+d-e\choose d-e}(n+1)-1$.  

The following lemma shows how to identify $\birn_d$ with a locally
closed subset of $\P_{(d,n,n)}$.  In particular,  $\birn_{d}$ is a
quasi-projective variety and  $\birn_{\leq d}$
is a finite union of quasi-projective varieties. The reader should be
aware that the topology induced by  $\birn$ on  $\birn_{\leq d}$ is
not the one given by this union.% see Theorem \ref{thmbirndtop}

\begin{lem}\label{birndtopo}
There exists a canonical bijection between $\birn_{d}$  and a locally
closed subset of $\P_{(d,n,n)}$. In particular,  $\birn_{d}$ is a
quasi-projective variety.
\end{lem}

\begin{proof}
Let $e<d$ be a non negative integer number. Consider the projective
spaces  $\P_{(d,n,n)}$,
$\P_{(d-e,n,n)}$ and $\P_{(e,n,0)}$.  Then there exists a  ``Segre
type'' morphism $s:\P_{(d-e,n,n)}\times \P_{(e,n,0)}\to \P_{(d,n,n)}$ which to a pair of
elements $(g_0:\cdots:g_n)\in \P_{(d-e,n,n)}$, $(f)\in \P_{(e,n,0)}$ it associates
$(g_0f:\cdots :g_nf)$. We denote by $\calw_e\subset \P_{(d,n,n)}$ the image of
$s$, which is a projective subvariety.

% Fix a basis for $\Bbbk[x_0,\ldots, x_n]_d$ given by monomials; it
% induces a basis for $\Bbbk[x_0,\ldots, x_n]_d^{n+1}$. If $p\in \P^N$
% is the point defined by a dimension 1 subspace generated by an element
% $(f_0,f_1,\ldots, f_n)\in \Bbbk[x_0,\ldots,
% x_n]_d^{n+1}-\{(0,\ldots,0)\}$, the homogeneous coordinates of $p$
% relative to that basis are the coefficients of the $f_i'$s ordered in
% accordance with the base order; we denote $p=(f_0:f_1:\cdots:f_n)$.      

Now consider the open set $\calu\subset\P_{(d,n,n)}$ consisting of points
$(f_0:f_1:\cdots:f_n)$ where the Jacobian determinant $\partial
(f_0,f_1,\ldots, f_n)/\partial (x_0,\ldots, x_n)$ is not identically
zero. Clearly, an element $(f_0:f_1:\cdots:f_n)\in \P_{(d,n,n)}\cap \calu$ can
be identified with a dominant rational map $\P^n\to \P^n$ defined by
homogeneous polynomials (without common factors) of degree $\leq
d$, and any such dominant rational map can be described in this way. Under this identification, points  in
$\calu_d:=\left[\P_{(d,n,n)}\backslash\left(\cup_{e=1}^{d-1}
    \calw_e\right)\right]\cap \calu$ are in one-to-one correspondence
with dominant rational 
maps defined by polynomials of degree exactly $d$.   

As it follows readily from \cite[Annexe B, Pro. B]{RPV}, the
(bijective) image of 
$\birn_d$ under the correspondence above is  closed  in $\calu_d$. Hence it is a
quasi-projective variety.
\end{proof}

The topology given by the preceding construction coincides with the
Zariski topology, inducing a structure of algebraic variety on $\birn_d$:

\begin{thm}[Blanc and Furter]\label{thm4.4}
 Let $\varphi:T\to \birn$ be a morphism with $d=\Deg(\varphi)$, and let $U_\varphi$ be as
in Proposition \ref{pro4.1}. Then we have:

\noindent
(a) the induced map
$U_\varphi\to  \birn_{d}$ is a morphism of algebraic varieties.

\noindent
(b)  the topology on $\birn_d$ induced by $\birn$ coincides with  the
topology of $\birn_d$ induced by $\P_{(d,n,n)}$
 as in Lemma \ref{birndtopo}. 
\end{thm} 
\qed

\subsection{Chevalley type Theorem}

\begin{thm}\label{thm:chevalley}
Let $X$ be a rational variety. If $\varphi:T\to \bir(X)$ is a morphism
and $C\subset T$ is a constructible set, then $\varphi(C)$ is
constructible and contains a dense open subset of $\overline{\varphi(C)}$.   
\end{thm}

\begin{proof}
By Lemma \ref{lem:functoriality} we may suppose $X=\P^n$ and $\varphi$
with degree $d=\Deg(\varphi)$. Hence $\overline{\varphi(T)}\subset
\birn_{\leq d}$; we consider the morphism
$\varphi_0:U_0=U_\varphi\to\birn_d$ induced by $\varphi$.  

On the other hand, Theorem \ref{thm4.2} gives a stratification
$T\backslash U_0=\cup V_j^{\ell}$ by locally closed sets such that
$d_j:=\deg\bigl(\varphi(t)\bigr)$ is constant on each $V_j$; set
$\varphi_j:V_j\to \birn_{d_j}$ the morphism induced by $\varphi$ on
$V_j$. 
   
We deduce that $\varphi(C)$ is constructible  by using Theorem
\ref{thm4.4} and  applying the 
standard Chevalley Theorem to the morphisms
$\varphi_0,\varphi_1,\ldots, \varphi_\ell$.  The last assertion of the
theorem is a  general topology result: since $\varphi(C)$ is
constructible, then $\varphi(C)=\cup_{i=1}^\ell Z_i$, where $Z_i$ is a locally
closed subset for all $i=1,\dots , \ell$. Then
\[
\varphi(C)\setminus \cup_i \bigl(\overline{Z_i}\setminus Z_i\bigr)=
\overline{\varphi(C)} \setminus \cup_i \bigl(\overline{Z_i}\setminus Z_i\bigr)
\]  
is a dense open subset of   $\overline{\varphi(C)}$.
\end{proof}

\subsection{Cyclic closed subgroups}

\begin{cor}
Let $\{f_m\}\subset \birn$ be a infinite sequence of birational
maps. Then  $\{f_m\}$ is closed if and only if 
$\lim_{m\to\infty}\deg(f_m)=\infty$. In particular, the Zariski
topology on $\bir(\PP^n)$ is not Noetherian. 
\end{cor}

\begin{proof}
Let $\varphi :T\to \birn$ be a morphism, with $\Deg(\varphi)=d$. Then
there exists  $m_0$  such that $\deg(f_m)\geq d$ for all $m\geq m_0$, and
thus $\varphi^{-1}\bigl(\{f_m\}\bigr)$ is finite. Hence, the if
follows from   Corollary \ref{cor4.3} and  Theorem \ref{thm4.4}. 

Conversely, suppose that
$\liminf_{m\to\infty}\deg(f_m)=d<\infty$. Then there exist infinitely 
many $f_i$ whose degree is $d$. Hence, $\{f_m\}\cap \bir(\PP^n)_{d}$ is
an infinite countable subset of the algebraic variety and thus it is  not closed. 
\end{proof}

\begin{cor}\label{cor:discret_subgroups}
Let $f\in\birn$ be a birational map of degree $d$. The cyclic subgroup $\langle f\rangle$ generated by $f$ is closed if and only if either $f$ is of finite order or $\lim_{m\to\infty}\deg(f^m)=\infty$. \qed
\end{cor}

When $n=2$ the behavior of the sequence $\deg(f^m)$ is well
known. Applying Corollary \ref{cor:discret_subgroups} one can thus
characterize when $\langle f \rangle$ is closed in $\bir (\PP^2)$.

\begin{cor}
Let $f\in \bir(\PP^2)$. Then the following assertions are equivalent:
\begin{enumerate}
\item $\langle f\rangle$ is not closed in $\bir(\PP^2)$.

\item $f$ has infinite order and $\bigl\{ \deg(f^m)\bigr\}_{m\in\N}$
  is  bounded.

\item $f$ is conjugated to an element of infinite order of
  $\operatorname{PGL}(3,\C)$. 
\end{enumerate}
\end{cor}
\begin{proof}
 Indeed,  following \cite{DiFa}, if $\langle f\rangle$ is infinite,
 then the sequence $\deg(f^m)$  either is bounded or grows with order
 at least $m$.
Hence, the infinite cyclic group $\langle f\rangle$ is not closed only
when the sequence  $\deg(f^m)$ is bounded.
The remaining equivalence follows from \cite[Thm. A]{BlDe13}.
\end{proof}

% 

% (a)  $\deg(f^m)\leq b$ for some positive $b\in \R$.

% (b)  $am\leq \deg(f^m)\leq bm$ for some positive $a,b\in\R$.

% (c)   $ am^2\leq \deg(f^m)\leq bm^2$ for some positive $a,b\in\R$.

% (d)  $ am^d\leq \deg(f^m)\leq bm^d$ for some positive $a,b\in\R$ where $d=\deg(f)$.

\subsection{Some big closed subgroups}\ %

Let $o\in\PP^n$ be a point. Consider the subgroup $\staro\subset
\birn$ of birational transformations which stabilize (birationality)
the set of lines passing through $o$. If $o'$ is another point
$\staro$ and $\star_{o'}(\PP^n)$ may be conjugated by mean of a linear
automorphism; in the sequel we fix $o=(1:0:\cdots:0)$. In \cite{Do}
the group $\staro$ is introduced in a different form and is called the
\emph{de Jonqui\`eres subgroup of level $n-1$} (see also \cite{Pa}).  

Let $\pi:\PP^n\tor\PP^{n-1}$ be the projection of center $o$ defined by
\[
(x_0:x_1:\cdots:x_n)\mapsto (x_1:\cdots:x_n).
\]
Then $\staro=\{f\in\birn: \exists \tau\in \bir(\PP^{n-1})\,, \pi
f=\tau\pi\}$. Moreover, note that  
 $\staro$ is the semidirect product 
\[\xymatrix{1\ar@{->}[r]&\jono\ar@{->}[r]&\staro\ar@{->}[r]^\rho&\bir(\PP^{n-1})\ar@{->}[r]&1}
\]
where $\jono=\{f\in\birn: \pi f=\pi\}$ and $\rho$ is the evident
homomorphism, and $\tau=\rho(f)$.
Indeed, the morphism  $\sigma:\bir(\PP^{n-1})\to\bir(\PP^n)$ given by  
\[(h_1:\cdots:h_n)\mapsto (x_0h_1:x_1h_1:\cdots:x_1h_n)\]
is injective and such that $\sigma\bigl( \bir(\PP^{n-1})\bigr)\subset \staro$. Clearly,    $\rho{\scriptstyle 
  \circ}\sigma=id$.

Moreover, we affirm that $\rho$ is continuous, and $\sigma $ is a
continuous closed immersion. Indeed, if  $\varphi:T\to \birn$ is a
morphism then the
composition $\rho{\scriptstyle \circ}\varphi$ defines a morphism
$T\to\bir(\PP^{n-1})$; therefore $\rho$ is a continuous function. 
Clearly, $\sigma$ is continuous. In order to 
prove, among other things,  that $\sigma$ is a closed immersion we
need the following:
   
\begin{lem}\label{lem*}
Let $f\in \Bbbk[T]\otimes \Bbbk[x_0,\ldots,x_n]$ be a  polynomial,
homogeneous  in $\x$; denote by $\deg_{x_0}(f)$ its degree in
$x_0$. Then for all integer $m\geq 0$ and $i=0,\ldots,n$ the sets 
\[
R= \bigl\{t\in T\mathrel{:} x_i| f(t,\x)\bigr\}\ ,\ S_m= \bigl\{t\in T; \deg_{x_0}(f) \leq m\bigr\}\]  
are closed in $T$.
\end{lem}

\begin{proof}
Let $a_1,\ldots,a_N\in \Bbbk[T]$ be the coefficients of $f$ as
polynomial in $x_0,\ldots,x_n$. It is clear that $R$ and $S_m$ are
defined as common zeroes of a subset of the polynomials 
$\{a_1,\ldots,a_N\}\subset \Bbbk[T]$.    
\end{proof}

\begin{thm}\label{thmfinal}
The subgroups  $\jono$ and $\staro$ are closed and $\sigma\bigl(\bir
(\PP^{n-1})\bigr)$ is closed 
in $\birn$. In particular, $\sigma$ is a closed immersion.
\end{thm}

\begin{proof}
Let $\varphi:T\to\birn$ be a morphism, say with $\Deg(\varphi)=d$. In
order to prove that $\varphi^{-1}\bigl(\jono\bigr)$  is closed it suffices to
consider a net $(t_\xi)$ in $\varphi^{-1}\bigl(\jono\bigr)$, where $\xi$ varies in a
directed set, and show that every limit point $t_\infty\in T$ of that
net satisfies $\varphi(t_\infty)\in \jono$. Let $t_\infty$ be such a
limit point and $T=\cup_{j=0}^lV_j$ be the 
stratification given by Theorem \ref{thm4.2}(a), where
$V_0=U_\varphi$ is the open set introduced in Proposition \ref{pro4.1}. Then there exists $j$ such that the subnet $(t_\xi)\cap
 V_j$ has $t_\infty$ as limit point. Thus, we  can  assume $t_\xi\in U_\varphi$ for all $\xi$, that is that $\deg(\varphi_{t_\xi})=d$. By shrinking $T$, if necessary, we may assume 
\[\varphi(t,\x)=\bigl(f_0(t,\x):\cdots:f_n(t,\x)\bigr),\]
where  $f_i\in \Bbbk[T]\otimes \Bbbk[\x]$ are homogeneous in
$\x=\{x_0,\ldots,x_n\}$ of degree $d$ (see Corollary \ref{corowritngdeg}(b)). From the description given in
\cite[\S 2]{Pa} it follows that for all $\xi$ there exists a
homogeneous polynomial $q_\xi\in \Bbbk[\x]$ such that: 
\begin{itemize}
\item[(a)]  $f_i(t_\xi,\x)=x_iq_\xi(\x)$, for $i>0$;
\item[(b)] $f_0(t_\xi,\x)$ and $q_\xi(\x)$ have degrees $\leq 1$ in $x_0$;
\item[(c)]  $f_0(t_\xi,\x)q_\xi(\x)$ has degree $\geq 1$ in $x_0$. 
\end{itemize}   

By Lemma \ref{lem*}, when $t_\xi$ specializes to $t_\infty$, then
$\varphi_\xi=\varphi(t_\xi)$ specializes to the birational map 
$\varphi_{t_\infty}=(f:x_1q:\cdots:x_nq):\PP^n\tor\PP^n$, where 
$f(\x)$ and $x_iq(\x)$, $i>0$, are polynomials in $\x$ of degree $d$
and with degree $\leq 1$ in $x_0$. Suppose that $f$ and $q$ admit a
common factor  $h\in \Bbbk[x_0,\ldots,x_n]$, of degree $\geq 0$. Since
the limit map $\varphi_{t_\infty}$
is birational (of degree $\leq d$) we deduce that $h\in
\Bbbk[x_1,\ldots,x_n]$: otherwise $h$ would have degree $1$ in $x_0$
and the map $\varphi_{t_\infty}$ would be defined by polynomials in
$x_1,\ldots,x_n$ contradicting birationality. Hence $f':=f/h$ and
$q':=q/h$ satisfy the conditions (b) and (c) above. We conclude
$\varphi_{t_\infty}=(f':x_1q':\cdots:x_nq')$. Applying again the
description of \cite[\S 2]{Pa}, we deduce that 
$\pi\varphi_{t_\infty}=\pi$ , that is  $\varphi_{t_\infty}\in\jono$,
which proves $\jono$ is closed.     

In order to prove that $\sigma\bigl(\bir(\PP^{n-1})\bigr)$ is  closed, consider a
net  $(t_\xi)\subset \varphi^{-1}\bigl(\sigma(\PP^{n-1})\bigr)$, with
limit 
point $t_\infty$. As
before, we can assume that $t_\xi \in U_\varphi$ for all $\xi$. With the notation
introduced above we have that
\begin{itemize}
\item[(a)]  $f_i(t_\xi,\x)=x_1h_{i,\xi}(\x)$, for $i>0$, and
\item[(b)] $f_0(t_\xi,\x)=x_0h_{1,\xi}(\x)$,
\end{itemize}   
where $\tau_\xi=(h_{1,\xi}:\cdots:h_{n,\xi}):\PP^{n-1}\tor\PP^{n-1}$ is
birational. From Lemma \ref{lem*} we obtain
that  $h_{i,\xi}$ specializes to a polynomial $h_i\in
\Bbbk[x_1,\ldots,x_n]$, $i>0$ and that 
$\varphi_{t_\infty}=(x_0h_1:x_1h_1:\cdots:x_1h_n)$. Since
$\pi\varphi_{t_\infty}=\tau_{\xi}\pi$ we conclude that
$\varphi_{t_\infty}\in\staro$ and thus $(h_1:\cdots:h_n)\in
\bir(\PP^{n-1})$ (\cite[Prop. 2.2]{Pa}).  Since 
$\sigma\bigl((h_1:\cdots:h_n)\bigr)=\varphi_{t_\infty}$, it follows
that $\sigma \bigl( \bir(\PP^{n-1})\bigr)$ is closed.   

Finally,  since for elements $f\in\jono$ and $h\in \bir(\PP^{n-1})$
the product $f\rtimes h$ is  the composition $f{\scriptstyle
  \circ}\sigma(h)$,  then 
$\staro=\jono Im(\sigma)$ (product in $\birn$). The fact that $\staro$
is closed follows then 
from the two assertions we have just proved together with the
continuity of the functions $\rho:\staro \to \bir\bigl(\PP^{n-1}\bigr)$, the
group product and the group 
inversion. Indeed, let  $(f_\xi\rtimes h_\xi)$ be a net in $\staro$
which specializes to $s\in\birn$. Then $\rho( f_\xi\rtimes
h_\xi)=\rho(1\rtimes h_\xi)=h_\xi$ specializes to  $\rho(s)=h\in
\bir(\PP^{n-1})$. Since $(f_\xi\rtimes h_\xi)\cdot (1\rtimes
h_\xi^{-1})=f_\xi\rtimes 1\in \jono$, the net $(f_\xi\rtimes 1)$
specializes to $s\sigma(h^{-1})\in\jono$. Thus $s\in\staro$. 
\end{proof}

\begin{rem}
More generally, for $\ell=1,\ldots, n$, the map $\sigma_\ell:\bir(\PP^{n-1})\to \birn$ defined by
 \[\sigma_\ell\bigl((h_1:\cdots:h_n)\bigr)= (x_0h_\ell:x_\ell h_1:\cdots:x_\ell h_n)\]
is a continuous, closed, homomorphism whose image is contained in
$\staro$ and such that $\rho\sigma_\ell=id$. In this notation, the map
$\sigma$ of Theorem \ref{thmfinal} is $\sigma_1$. Moreover, one has
\[
\bigcap_{\ell=1}^n \sigma_\ell\bigl(\bir(\PP^{n-1})\bigr)=\{id\}.
\]

If $\calu_\ell$ is the dense open set $\birn\backslash
\sigma_\ell\bigl(\bir(\PP^{n-1})\bigr)$, then $\birn-\{id\}=\bigcup_{\ell=1}^n \calu_\ell$.
\end{rem}

\end{document}